\documentclass[10pt]{article}
\usepackage[english]{babel}
\usepackage[latin9]{inputenc}
\usepackage[dvips]{graphics}
\usepackage{makeidx}
\usepackage{index}
\usepackage{amsfonts, amsthm,amssymb,amscd,amsmath}
\usepackage{stmaryrd}
\usepackage[all]{xy}
\newtheorem{Def}{Définition}[section]

\newtheorem{thm}[Def]{Théorème}
\newtheorem{prop}[Def]{Proposition}

\newtheorem{rk}[Def]{Remarque}

\newtheorem{conj}{Conjecture}

\newcommand{\A}{ \mathbb A }

\newcommand{\C}{ \mathbb C}
\newcommand{\F}{\mathbb F}
\newcommand{\G}{ \mathbb G }
\newcommand{\Q}{\mathbb Q}

\newcommand{\oo}{ \mathbb O}

\newcommand{\p}{\mathbb P}

\newcommand{\Z}{\mathbb Z}

\newcommand{\ca}{{\cal A}}

\newcommand{\cc}{{\cal C}}

\newcommand{\ce}{{\cal E}}
\newcommand{\cf}{{\cal F}}

\newcommand{\ck}{{\cal K}}
\newcommand{\cl}{{\cal L}}
\newcommand{\cm}{{\cal M}}

\newcommand{\co}{{\cal O}}
\newcommand{\cp}{{\cal P}}

\newcommand{\cv}{{\cal V}}

\newcommand{\im}{\hbox{im}}

\newcommand{\supp}{\hbox{supp}}

\newcommand{\cox}{{{\cal O}_X}}

\newcommand{\cob}{{\cal O}_{{\rm Pic}^dX\times \C^g}}

\newcommand{\coxxt}{{{\cal O}_{X\times_\C T}}}

\newcommand{\s}{\mathfrak{S}}

\newcommand{\noi}{\noindent}
\newcommand{\disp}{\displaystyle}

\newcommand{\commutatif}{\ar@{}[rd]|{\circlearrowleft}}

\newcommand{\eq}[1][r]
   {\ar@<-3pt>@{-}[#1]
    \ar@<-1pt>@{}[#1]|<{}="gauche"
    \ar@<+0pt>@{}[#1]|-{}="milieu"
    \ar@<+1pt>@{}[#1]|>{}="droite"
    \ar@/^2pt/@{-}"gauche";"milieu"
    \ar@/_2pt/@{-}"milieu";"droite"}
\def\fract#1/#2{\hbox{\leavevmode
  \kern.1em \raise .25ex \hbox{\the\scriptfont0 $#1$}\kern-.1em }\big/
  {\hbox{\kern-.15em \lower .5ex \hbox{\the\scriptfont0 $#2$}} }}

\def\ifract#1/#2{\hbox{\leavevmode
  \kern.1em \lower .5ex \hbox{\the\scriptfont0 $#1$}\kern-.1em }\big\backslash
  \hbox{\kern-.15em \raise .3ex \hbox{\the\scriptfont0 $#2$}} }
\def\iifract#1/#2{\hbox{\leavevmode
\kern.1em \lower .5ex \hbox{\the\scriptfont0 $#1$}\kern-.1em }\big\backslash\big\backslash
  \hbox{\kern-.15em \raise .3ex \hbox{\the\scriptfont0 $#2$}} }
\def\dblfract#1/#2/#3{\hbox{\leavevmode
  \kern.1em \lower .5ex \hbox{\the\scriptfont0 $#1$}\kern-.1em }\big\backslash
  \hbox{\kern-.15em \raise .3ex \hbox{\the\scriptfont0 $#2$}}\big/
  \hbox{\kern-.15em \lower .5ex \hbox{\the\scriptfont0 $#3$}} }

\def\bas#1{\hbox{\leavevmode
  \kern.1em \lower .8ex \hbox{\the\scriptfont0 $#1$}\kern-.1em }}
\def\bass#1{\hbox{\leavevmode
  \kern.1em \lower .2ex \hbox{\the\scriptfont0 $#1$}\kern-.1em }}

\def\addnotation #1:#2:#3{ \parbox{2.5cm}{#1}  \parbox{5.3in}{#2 \dotfill \pageref{#3}}\\}
\def\notation #1#2#3 {\parbox{2.5cm}{{#1}}\parbox{5.2in}{{#2}\dotfill {#3}}\\}

\title{\LARGE The geometric correspondence in some special cases}
\author{\normalsize C\'ecile Poirier}

\date{\normalsize}

\begin{document}
\maketitle
\begin{abstract}
The geometric Langlands correspondence for function fields over finite fields has been proved by Frenkel, Gaitsgory, Vilonen in \cite{FGV}. The aim of this article is to write translation for curves over $\C$ and prove the correspondence in some special cases and some new cases, outside the frame of the usual geometric Langlands correspondence.
\end{abstract}

\section{The geometric version of Langlands conjecture}
We start by recalling the geometric Langlands conjecture expressed in terms of Hecke operators. We refer to the book 'introduction to the Langlands program' \cite{BG} for this first part.
We consider a projective smooth curve $X/\F_q$ of genus $g>0$ which is geometrically irreducible, that is to say such that $X\times_{\F_q} {\overline\F_q}$ is irreducible. 
Note that we write $X\times_{\F_q} {\overline\F_q}$ instead of $X\times_{{\rm Spec}(\F_q)}{\rm Spec}( {\overline\F_q})$.
The stack of rank $n$ vector bundles is denoted by ${\rm Bun}_{n,X}$ and we denote by ${\rm Bun}_{n,X}(\F_q)$ (or simply ${\rm Bun}_n(\F_q)$) the set of its $\F_q$-points. Recall that an $\F_q$-point of the stack ${\rm Bun}_{n,X}$ is a morphism of stacks
\[\underline{{\rm Spec}(\F_q)}\rightarrow {\rm Bun}_{n,X}\]
Such a morphism is determined by the image of the identity map which is an object of ${\rm Bun}_{n,X}$ over $X\times Spec(\F_q)$ therefore an $\F_q$-point of ${\rm Bun}_{n,X}$ can be identified with a vector bundle of rank $n$ over $X$.

\smallskip

\noi For $i=1\ldots n$, one can define a stack $Hecke_i$ over $Sch_\C$ such that for any scheme $T$ over $\F_q$ $Hecke_i(T)$ denote the set of tuples $(\cm,\cm',x)$, where  $x$ is closed point of $X$ and $\cm,\cm'$ are vector bundles of rank $n$ over $X\times_{\F_q}T$ satisfying $\cm'\subset\cm$ and $\fract \cm/{\cm'}\simeq\left({\fract \co_{X\times_{\F_q}T}/{\co_{X\times_{\F_q}T}(-[{x}\times T])}}\right)^i$.\\
\noi(Cf \cite{Lau} or \cite{Fre})

\vskip0.3cm

\noi The $i$-th Hecke correspondence is given by:
$$\xymatrix{&Hecke_i \ar[rd]^{supp\times h^{\rightarrow}}\ar[ld]_{h^{\leftarrow}}&\\ {\rm Bun}_{n,X}& & X\times {\rm Bun}_{n,X}}$$
where $h^{\leftarrow}(\cm,\cm',x)=(\cm)$, $h^{\rightarrow}(\cm,\cm',x)=\cm'$ and $\supp(\cm,\cm',x)=x$.
Let $x\in|X|$ and write $Hecke_{i,x}=supp^{-1}(x)$. This gives a correspondence between $\{x\}\times {\rm Bun}_{n,X}$ and ${\rm Bun}_{n,X}$. We can show that the double quotient 
$$\dblfract {\rm GL}_n(F)/{\rm GL}_n(\A)/{{\rm GL}_n(\oo)}$$
is in bijection with the set of (equivalence classes of) $\F_q$-points of ${\rm Bun}_{n,X}$. Therefore the diagram defines an operator on the space of functions on $\dblfract {\rm GL}_n(F)/{\rm GL}_n(\A)/{{\rm GL}_n(\oo)}$ which associates to $f$ the function $h^{\rightarrow}_!((h^{\leftarrow})^*f)$. The operator $h^{\rightarrow}_!$ is the integration of the function along the fibres of $h^{\rightarrow}$. We can check that this operator is precisely the $i$-th Hecke operator $H_{i,x}$. The (perverse) sheaves corresponding to the automorphic functions associated to the automorphic representations (which occur in the ordinary Langlands correspondence) are therefore the eigensheaves of the Hecke correspondence (see the appendix \ref{langlands} for more details ok?)

\smallskip

\noi Let $\ck$ be a perverse sheaf on ${\rm Bun}_{n,X}$ and let $H_i(\ck)=(supp\times h^{\rightarrow})_!h^{\leftarrow*}(\ck)$, where $(supp,h^\rightarrow)_!$ denotes the pushforward with compact support. This is again a perverse sheaf on ${\rm Bun}_{n,X}$. The geometric form of the Langlands correspondence can be stated as follows.
\begin{thm}
Let $X$ be an irreducible smooth projective curve over $\F_q$. For any irreducible local system $E$ on $X$ of rank $n$, there exists a (unique) perverse sheaf $\ck$ on ${\rm Bun}_{n,X}$, irreducible on every connected component of ${\rm Bun}_{n,X}$, such that $\forall i=1...n$, we have $H_i(\ck)\simeq (\wedge^i E)\boxtimes \ck$
\end{thm}
\smallskip
\noi (Cf \cite{Fre}, \cite{Lau} and \cite{FGV})

\noi The authors E. Frenkel, D. Gaitsgory and K.Vilonen assert that an analogous theorem holds when the finite field $\F_q$ is replaced by any field of zero characteristic. Our aims are the following: 

The aim of this article is to find a consistent translation of the theorem over $\C$ and to give proofs for some special cases.

We write $\widehat{ \pi_1(X)}$ for the completion of the fundamental group. Note that local systems occurring in the geometric Langlands conjecture correspond to continuous homomorphisms 
$$\rho:\widehat{ \pi_1(X)}\rightarrow {\rm GL}_n(\overline {\Q_l})$$
for a certain prime $l$ (and such that the image of this homomorphism is in fact contained in ${\rm GL}_n(E)$ where $E$ is a finite extension of $\Q_l$).
Note that we can identify $\overline{\Q_l}$ with $\C$ but there is no canonical way to do so.  For the translation in the complex case, we consider representations
$$\rho':\pi_1(X)\rightarrow {\rm GL}_n(\C)$$ which induces a (complex) local system.

\noindent Not all the homomorphisms 
$$\tau:\pi_1(X)\rightarrow {\rm GL}_n(\C)\simeq {\rm GL}_n(\overline{\Q_l})$$
do extend to a continuous homomorphism $\widehat{ \pi_1(X)}\rightarrow {\rm GL}_n(\overline {\Q_l})$. In this thesis, we choose to work with the largest class of complex local systems which corresponds to representations of $\pi_1(X)$. 
There exists a corresponding notion of perverse sheaf on a complex variety which is needed for the general case. 
\begin{conj}
Let $X$ be a smooth irreducible projective curve over $\C$ and let $E$ be an irreducible local system of rank $n$ over $X$. Then there exists a (unique) perverse sheaf $\ck$ on ${\rm Bun}_{n,X}$ irreducible on every connected component of ${\rm Bun}_{n,X}$ such that $\forall i=1...n$, we have $H_i(\ck)\simeq (\wedge^i E)\boxtimes \ck$.
\end{conj}
We will be interested in special cases of the theorem where only local systems on the set ${\rm Bun}_{n,X}(\C)$ of $\C$-points of the stack ${\rm Bun}_{n,X}$ are involved, in particular the case $n=1$. We have seen that the category of local systems is equivalent to the category of connections and therefore we have a similar statement when replacing the local system $E$ by a connection $(\cv,\nabla)$.

\smallskip 

In the case $n=1$, we will state and prove the theorem for the algebraic variety ${\rm Pic}X$ instead of the set ${\rm Bun}_{1,X}(\C)$ of $\C$-points of the stack ${\rm Bun}_{1,X}$ and connections, replacing complex local systems. Replacing ${\rm Bun_{1,X}}$ by ${\rm Pic}X$ is legitimate since the first set is the set of line bundles on $X$ and the second is the set of equivalence classes of line bundles on $X$.

\section{The $\G_m$ case}
In this section, we are interested in connections of rank 1 on the curve $X$ of genus $g>0$. Such a connection has a differential Galois group contained in $\G_m$ (Cf \cite{vdPS}). 

Because all the vector bundles involved are of rank 1, only the stack ${\it Hecke_1}$ is non empty. Moreover, the map ${\it Hecke_1(\C)}\rightarrow X\times {\rm Bun}_{1,X}(\C)$ is an isomorphism. Indeed, if $\cm'\subset\cm$ are two line bundles satisfying $\fract\cm/\cm'\simeq \fract\cox/{\cox(-[x])}$, we have $\cm(-[x])\subset \cm'\subset \cm$ and this imposes $\cm'=\cm(-[x])$.\\
Thus 
$$Hecke_1(\C)=\{(\cm,\cm(-[x]),x), x\in|X|\hbox{ and }\cm \hbox{ line bundle} \}.$$
and we have the map 
$$H_1:X\times {\rm Bun}_{1,X}\rightarrow {\rm Bun}_{1,X},$$ given by $H_1(x,\cl)=\cl([x])$. We write again $H_1$ for the map
$$H_1:X\times {\rm Pic}X\rightarrow {\rm Pic}X,$$ given by $H_1(x,[\cl])=[\cl([x])]$.

\begin{prop}\label{casgm}
Let $(\ce,\nabla)$ be a rank 1 connection on $X$. Then there exists a unique (integrable) connection $(\ck,\nabla_\ck)$ of rank 1 on ${\rm Pic}X$ satisfying $H_1^*\ck\simeq \ce\boxtimes \ck$, where the isomorphism denotes an isomorphism of connections.
\end{prop}

\noindent To construct $\ck$, we imitate Deligne's proof of the Langlands geometric correspondence for ${\rm GL}_1$ ( Cf \cite{Fre}). Denote by $\ck_d$ the restriction of $\ck$ to ${\rm Pic}^d X$. Let $n$ denote an integer such that $n>2g-2$ where $g$ is the genus of $X$. 
\subsection{Preliminary results}
As before, $X$ denotes a smooth projective irreducible curve over $\C$.
Let $X^n$ denote the product of $n$ copies of $X$ where we assume that $n>2g-2$. On this algebraic variety of dimension $n$, there an obvious connection $\ce^n:=\ce\boxtimes...\boxtimes \ce$ of rank 1 associated to the connection $\ce$ on $X$. The group of permutations $\s_n$ clearly acts on $X^n$ and $ \ce^n$. This induces an integrable connection $\ce^{(n)}$ on the quotient $X^{(n)}=X^n/\s_n$ which is non singular (Cf \cite{Mil}). This connection is defined locally by 
\[\ce^{(n)}(\pi_n(U)):=\ce^n(\pi_n^{-1}(\pi_n(U))^{\s_n}\] 
where $\pi_n:X\rightarrow X^{(n)}$ and $U$ is an open of $X^n$. In other words, $\ce^{(n)}$ is defined locally as the elements of $\ce^n$ invariant under $\s_n$.
Let $B_n$ denote the natural map $X\times X^n\rightarrow X^{n+1}$. Write $(\overline{x_1,..,x_n})$ the image of $(x_1,..,x_n)$ in $X^n/\s_n$, and
$$\begin{array}{cccc}
\beta_n:& X\times X^{(n)}&\rightarrow & X^{(n+1)}\\ & (x,(\overline{x_1,..,x_n}))& \mapsto & (\overline{x,x_1,..,x_n})
\end{array}$$

\vskip0.2cm 
One can check that $\beta_n^*\ce^{(n+1)}\simeq \ce\boxtimes \ce^{(n)}$.

Let $\rho$ stand for the representation of $\pi_1(X)$ associated to the connection $\ce$. Recall that we omit to mention the base point in the fundamental groups for convenience.

\vskip0.2cm

\noindent Now consider the natural morphism $f_n:X^{(n)}\rightarrow {\rm Pic}^n X$ which associates the class of the divisor $[x_1]+...+[x_n]$ to the $n$-tuple $(\overline{x_1,...,x_n})\in X^{(n)}$.
The fibres of this morphism are projective spaces of dimension $n-g$. Moreover, because $f_n:X^{(n)}\rightarrow {\rm Pic}^nX$ is a fibration for $n>2g-2$, it gives a long exact sequence of higher homotopy groups.
\begin{multline*}
\ldots\rightarrow \pi_2({\rm Pic}^nX)\rightarrow \pi_1(\p^{n-g})\rightarrow \pi_1(X^{(n)})\stackrel{p}{\rightarrow}\\
\pi_1({\rm Pic}^nX)\rightarrow \pi_0(\p^{n-g})\rightarrow \pi_0(X^{(n)})\rightarrow\pi_0({\rm Pic}^nX)
\end{multline*}
We know that $\pi_0(\p^{n-g})$ is trivial because $\p^{n-g}$ is connected. Besides, ${\rm Pic}^{n}X$ is an aspherical space, which means that all its higher homotopy groups $\pi_i({\rm Pic}^{n} X)$ for $i\geq 2$ are trivial. This can be shown by saying that ${\rm Pic}^0X$ admits $\C^{g}$ as universal cover and the higher homotopy groups (for $i\geq 2$) of ${\rm Pic}^0 X$ and $\C^g$ are isomorphic (Cf \cite{Gre}). Since $\pi_i(\C^{g})$ is trivial for all $i$ and $\pi_1(F)$ is also trivial, the long exact sequence is simply an isomorphism between $\pi_1(X^{(n)})$ and $\pi_1({\rm Pic}^nX)$. 

\subsection{Construction of the connection}
We have seen that we have an isomorphism $$\widetilde f_n:\pi_1(X^{(n)})\stackrel{\sim}{\longrightarrow } \pi_1({\rm Pic}^n X)$$ 
given by the map induced by $f_n$.
\hfill\break
We write $(\ck_n,\nabla_n)$ for the connection on ${\rm Pic}^n X$ we get from $\ce^{(n)}$ through this isomorphism. More precisely, the connection $\ce^{(n)}$ is associated to a representation $\rho_n$ of $\pi_1(X^{(n)})$. Thanks to this isomorphism, we can define a representation $\eta_n$ of $\pi_1({\rm Pic}^n X)$. The latter is associated to a connection we denote by $(\ck_n,\nabla_n)$ which is unique up to isomorphism. It satisfies $f_n^*\ck_n\cong \ce^{(n)}$ (isomorphism of connections) by equivalence between the category of connections on $X$ and the category of representations of $\pi_1(X)$.

It remains to prove that $$H_{1,n}^*\ck_{n+1}\simeq \ck_n\boxtimes \ce.$$ 
\noindent We will use the fact that the following diagram is commutative:
$$\xymatrix{X\times X^{(n)} \ar[rr]^{\beta_n} \ar[dd]_{(Id,f_n)}&& X^{(n+1)}\ar[dd]^{f_{n+1}}\\ &&\\
  X\times {\rm Pic}^n X \ar[rr]_{H_{1,n}} && {\rm Pic}^{n+1}X }$$

\vskip0.2cm
In order to prove that it implies $H_{1,n}^*\ck_{n+1}\simeq \ce\boxtimes \ck_n$, one prove that the representations associated are isomorphic.

\noindent It remains to define the connection $\ck_n$ when $n\leq 2g-2$. We will construct it using  $\ck_{2g-1}$. The connection $H_{1,2g-1}^* \ck_{2g-1}$ is integrable on $X\times {\rm Pic}^{2g-2}X$. It corresponds to a representation of
$$\pi_1(X\times {\rm Pic}^{2g-2}X)=\pi_1(X)\times \pi_1({\rm Pic}^{2g-2}X),$$
that we can decompose as the product of a representation of $\pi_1(X)$ and a representation of $\pi_1({\rm Pic}^{2g-2}X)$. Then there exists $(\cf,\nabla_\cf)$ and $(\ck_{2g-2},\nabla_{2g-2})$, two integrable connections respectively on $X$ and ${\rm Pic}^{2g-2}X$, corresponding to those two representations. The external product $\cf\boxtimes \ck_{2g-2}$ is a connection on $X\times {\rm Pic}^{2g-2}X$, so by uniqueness, we have the isomorphism of connections 
$$ H_{1,2g-1}^* \ck_{2g-1}\cong \cf\boxtimes \ck_{2g-2}.$$ 

\noindent It remains to prove that $\cf\cong \ce$. For this purpose, we consider the map
$$\begin{array}{cccc}
 f: & X\times X\times {\rm Pic}^{2g-2}&\rightarrow &{\rm Pic}^{2g}X \cr 
 & (x,y,{\cal L})&\mapsto& {\cal L}([x]+[y])\end{array}$$
and we decompose $f$ in two different ways
\begin{enumerate}
\item $f:(x,y,{\cal L})\mapsto (x,{\cal L}([y]))\mapsto {\cal L}([x]+[y])$

\item $f:(x,y,{\cal L})\mapsto (y,{\cal L}([x]))\mapsto {\cal L}([x]+[y])$
\end{enumerate}

We have thus
\begin{enumerate}
\item $\disp f^* \ck_{2g}=\ce\boxtimes \cf\boxtimes \ck_{2g-1}$ and
\item $\disp f^* \ck_{2g}=\cf\boxtimes \ce\boxtimes \ck_{2g-1}$
\end{enumerate}
By uniqueness, we have $\cf\cong \ce$, so $H_{1,2g-1}^* \ck_{2g-1}\cong \ce\boxtimes \ck_{2g-2}$.
We define by descending recursion all the $\ck_n$ which verify (by construction) $$H^*_{1,n}\ck_{n+1}\cong \ce\boxtimes \ck_{n}, \forall n\in \Z \hbox{ hence }$$ 
$$H_1^*\ck\cong \ce\boxtimes \ck$$ 
and this ends the proof.
\section{The $\G_a$-case}

The group $\G_a$ is identified with the subgroup 
$$ \left\{\begin{pmatrix} 1& a \\ 0&1 \end{pmatrix}, a \in \C \right\}$$
of ${\rm GL}_2(\C)$ and we consider the rank 2 connections on $X$ with differential Galois group contained in $\G_a$.
These connections satisfy the short exact sequence (of connections) 
$$0\rightarrow \cox\rightarrow (\ce,\nabla)\rightarrow \cox\rightarrow 0$$
where $\cox$ is provided with the trivial connection.

\noi To prove this, we consider the representation $\rho:\pi_1(X)\rightarrow {\rm GL}_2$ associated to $(\ce,\nabla)$. Such a representation can be identified with a $\pi_1(X)$-module, namely with $\C^2$ provided with the action of $\pi_1(X)$ induced by $\rho$.
The Zariski closure of the image of $\rho$ is equal to the differential Galois group of the connection $(\ce,\nabla)$ so it lies in $\G_a$, and we have in particular $\im \rho \subset \G_a$.
The representation $\rho$, or more exactly the associated $\pi_1(X)$-module, admits a short exact sequence of $\pi_1(X)$-modules:
\[0\rightarrow \C\rightarrow \C^2\rightarrow \C\rightarrow 0\]
where $\C$ is the one-dimensional trivial $\pi_1(X)$-module. 

Hence, we obtain an exact sequence of representations of $\pi_1(X)$, the extreme terms of the sequence being trivial representations. By the equivalence of categories stated in Section \ref{equiv} we therefore have an exact sequence of connections. The trivial representations correspond to the sheaf $\cox$ provided with the trivial connection and this proves that such a connection verifies 
$$0\rightarrow \cox\rightarrow (\ce,\nabla)\rightarrow \cox\rightarrow 0$$
where $\cox$ is provided with the trivial connection.

The substack ${\rm Bun}'$ of ${\rm Bun}_2$ we will consider in order to deal with those connections is the stack of fibre bundles $\cm$ of rank 2 over $X\times_\C T$ admitting an exact sequence of the form 
$$0\rightarrow \coxxt\rightarrow \cm\rightarrow \coxxt\rightarrow 0$$
Note that the set ${\rm Bun}'(\C)$ of $\C$-points of ${\rm Bun'}$ can be seen as the affine space 
$${\rm Ext}^1_{\cox}(\co_X,\co_X)=H^1(X,\co_X)\cong \C^g.$$ 
We also note that there is no Hecke correspondence in this situation. Indeed if we consider $Hecke_1(\C)$ or $Hecke_2(\C)$, we have 
$$\cm(-[x])\subset \cm'\subset \cm,$$
because $\fract\cm/{\cm'}\cong (\fract\co/{\cox(-[x])})^i$.

\begin{itemize}
\item For $i=2$, this imposes $\cm'=\cm(-[x])$, 
\item For $i=1$, this implies $\deg \cm'<\deg \cm=0$ so in this case again, $\cm'$ cannot be an extension of $\cox$ by $\cox$.
\end{itemize}
We are therefore forced to consider a bigger subgroup and a bigger family of connections of rank 2.

\noindent Consider the subgroup 
$$\left\{ \begin{pmatrix}a & b\\ 0 & a\\ \end{pmatrix}, a \in \C^*,b\in\C \right\}$$
 which is $\G_m\times\G_a$.
The family of connections on $X$ we need to consider are those verifying an exact sequence of the form
$$0\rightarrow (\cp,\nabla_\cp)\rightarrow (\cm,\nabla)\rightarrow (\cp,\nabla_\cp) \rightarrow 0.$$
The corresponding stack ${\rm Bun}'_2$ is the stack of vector bundles satisfying such a sequence (without the connections). More precisely, ${\rm Bun}'_2(T)$ will denote the set of vector bundles $\cm$ of rank 2 over $X\times_\C T$ satisfying an exact sequence 
$$0\rightarrow \cm_1\rightarrow \cm\rightarrow \cm_1 \rightarrow 0$$ 
where $\cm_1$ is a line bundle on $X\times_\C T$. It is on the set of $\C$-points of the stack ${\rm Bun}_2'$ that we want to prove the Hecke correspondence. 

\noindent The only Hecke correspondences possible are $Hecke_1$ and $Hecke_2$. We observe that the vector bundle appearing in the centre of the exact sequence, and element of ${\rm Bun}'_2$, must have an even degree, equal to  $2\deg \cm_1$ if $\cm$ satisfies 
\[0\rightarrow \cm_1\rightarrow \cm\rightarrow \cm_1\rightarrow 0\]

\vskip0.5cm
\noi In the case of $Hecke_1$ (or more precisely $Hecke_i(\C)$), we have $\fract\cm/{\cm' }\simeq \fract\cox/{\cox(-[x])}$ so 
$$\cm(-[x])\subset \cm'\subset \cm$$
and because $\cm$ is of rank 2, we also have 
$$\fract\cm/{\cm(-[x])}\simeq (\fract\cox/{\cox(-[x])})^2,$$
therefore $\cm(-[x])\subsetneq \cm'\subsetneq \cm$, and thus $\deg \cm(-[x])<\deg \cm'<\deg \cm$, which imposes $\deg \cm'= \deg \cm-1$. Since the degrees cannot be simultaneously even, the only correspondence which can occur is $Hecke_2$.
\smallskip
Let us show that in this case if $\cm',\cm$ are elements of $ {\rm Bun}'_2(\C)$ and  $\fract\cm/{\cm'}\cong (\fract\cox/{\cox(-[x])})$ then this imposes $\cm'=\cm(-[x])$. As before, we have $$\cm(-[x])\subset \cm'\subset \cm.$$ 
Because the sheaf $\cm$  has rank 2, we have 
$$\fract\cm/{\cm(-[x])}\cong (\fract\cox /{\cox(-[x])})^2$$ and this imposes $\cm'=\cm(-[x])$. We are in a situation similar to the previous one, namely we have an isomorphism between $Hecke_2(\C)$ and the set $X\times {\rm Bun}'_2(\C)$ which gives a morphism 
\[ H_2:\left\{\begin{array}{ccc} 
X\times {\rm Bun}'_2(\C)&\rightarrow& {\rm Bun}'_2(\C) \\ 
(x,\cm)&\mapsto& \cm([x])
\end{array} \right.\]
Instead of considering the set ${\rm Bun}'_2(\C)$, we will work with the algebraic variety ${\rm Pic} X\times \C^g$ since as for the classical case, we have a map between the two sets given by 
$$(\cm \hbox{ with }{0\rightarrow \cm_1\rightarrow \cm\rightarrow \cm_1 \rightarrow 0})\mapsto ([\cm_1],{0\rightarrow \cox\rightarrow \cm_1^{-1}\otimes \cm\rightarrow \cox \rightarrow 0})$$
We write again $H_2$ for the morphism 
\[ H_2:\left\{\begin{array}{ccc}
X\times {\rm Pic} X\times \C^g&\rightarrow& {\rm Pic} X\times\C^g \\ 
(x,[\cm_1],v)&\mapsto& ([\cm_1([x])],v)
\end{array}
\right.\]
\noindent According to the previous identification for vector bundles, the connections we are interested in are given by a pair $(\cp,{0\rightarrow\cox\rightarrow \cc\rightarrow \cox \rightarrow 0})$ or simply $((\cp,\nabla_\cp),(\cc,\nabla_\cc))$, where $(\cp,\nabla_\cp)$ is a rank 1 connection and $(\cc,\nabla_\cc)$ is a rank 2 connection satisfying 
$$0\rightarrow\cox\rightarrow \cc\rightarrow \cox \rightarrow 0$$
and $\ce=\cp\otimes \cc$. We assume again that the sheaves $\cox$ are provided with the trivial connection. The connection $(\cc,\nabla_\cc)$ has therefore a Galois group included in $\G_a$. 

\begin{prop}
Let $\ce=((\cp,\nabla_\cp),(\cc,\nabla_\cc))$ be a connection of rank 2 on $X$ with differential Galois group included in $\G_a\times\G_m$. Then there exists an integrable connection $\ck$ of rank 2 on ${\rm  Pic} X\times  \C^g$ satisfying $H^*_2 \ck\cong \wedge^2 \ce\boxtimes \ck$ (isomorphism of connections).
\end{prop}
\begin{proof}
\noindent As before, we write 
$${\rm Pic} X\times \C^g=\bigcup {\rm Pic}^d X \times \C^g.$$ 
We will construct $\ck_d$, the restriction of $\ck$ to ${\rm Pic}^d X\times \C^g$. We search $\ck_d$ as a product $\cl_d\otimes \ca_d$ of a rank 1 connection $\cl_d$ and a connection $\ca_d$ of rank 2 and which satisfies an exact sequence of the form
$$0\rightarrow \cob \rightarrow \ca_d\rightarrow \cob \rightarrow0,$$ 
where the sheaf $\cob$ is provided with the trivial connection. Here the sheaf $\ck_d$ is provided with the connection $\nabla_{\cl_d}\otimes \nabla_{\ca_d}$. We will omit to specify the connections and only write the vector bundles.

\noindent According to the $\G_m$-case, we know there exists a connection $\tilde{\ck}$ on ${\rm Pic} X$ associated to the rank 1 connection $\cp^{\otimes 2}$ on $X$. Let $\cl:=\tilde{\ck}\boxtimes \co_{\C^g}$ denote the connection on ${\rm Pic}X\times\C^g$. 
\vskip0.1cm

Now we want to construct a rank 2 connection. Since $\C^g$ has a trivial fundamental group one has
$$\pi_1({\rm Pic}^1X\times\C^g)\cong \pi_1({\rm Pic}^1 X)\simeq\pi_1(X)_{ab}.$$ 
Moreover the differential Galois group of the connection $\cc$ is contained in $\G_a$ so the latter corresponds to a representation $\pi_1(X)\rightarrow \G_a$. By the above isomorphism, this representation induces a representation $\pi_1({\rm Pic}^{1}X\times \C^g)\rightarrow \G_a$ which is associated to a connection denoted by $\ca_1$ on ${\rm Pic}^1X\times\C^g$ .
Fix a point $P_0\in |X|$. We have natural isomorphisms 
$$t_d:{\rm Pic}^1X\times \C^g\rightarrow {\rm Pic}^{d+1}X\times\C^g$$ 
given by $(\cl,v)\mapsto (\cl(d[P_0]),v)$. These isomorphisms give isomorphisms denoted by $\widetilde{t_n}$ between the fundamental groups (and independent from the choice of $P_0$). Thus, for any integer $d$, we have an isomorphism $\pi_1(X)_{ab}\stackrel{\sim}{\rightarrow}\pi_1({\rm Pic}^{d}X\times \C^g)$ which gives a connection $\ca_d$ on ${\rm Pic}^dX\times\C^g$.
 
It remains to prove that the connection $\ck_d$ on ${\rm Pic}^dX\times\C^g$ defined by $\ck_d:=\cl_d\boxtimes \ca_d$ satisfies
$$H^*_2 \ck\cong \wedge^2 \ce\boxtimes \ck.$$
We consider the following map
$$H_{2,d}: X\times {\rm Pic}^d X\times \C^g\rightarrow  {\rm Pic}^{d+1} X\times\C^g.$$
Write ${\tilde \ck}_d$ for the restriction on ${\rm Pic}^d X$ of the connexion $\tilde \ck$ (which is associated to $\cp^{\otimes 2}$). By definition, we have 
$$\cl_d={\tilde \ck}_d\boxtimes \co_{\C^g}$$ 
where the sheaf $\co_{\C^g}$ is provided with the trivial connection. The sheaf $\cl_d$ then satisfies
$$H_{2,d}^*\cl_{d+1}=H_{1,d}^*{\tilde \ck}_{d+1}\boxtimes \co_{\C^g}\cong \cp^{\otimes 2}\boxtimes {\tilde \ck}_d\boxtimes \co_{\C^g}=\cp^{\otimes2}\boxtimes \cl_d.$$
Note that we have 
$$\wedge^2 \ce=\wedge^2(\cc\otimes \cp)=\wedge^2 \cc \otimes \cp^{\otimes 2}$$
because $\cp$ has rank 1. Moreover, $\cc$ admits a short exact sequence
\[0\rightarrow \cox\rightarrow \cc\rightarrow \cox\rightarrow 0\]
and thus we have $\wedge^2 \cc=\cox$.
It therefore suffices to prove that we have 
$$H_{2,d}^* \ca_{d+1}\cong \cox \boxtimes \ca_d$$
where on $\cox$ we have the trivial connection.
In order to prove this, we write the following commutative diagram
$$\xymatrix{ X\times {\rm Pic}^d X\times\C^g \ar[r]^{H_{2,d}}\ar[d]_{F}&{\rm Pic}^{d+1} X\times\C^g\\
  X\times {\rm Pic}^dX\times\C^g \ar[ru]_ { u_d\circ p_2}& }$$
where $p_2$ denotes the second projection $X\times {\rm Pic}^dX\times\C^g \rightarrow {\rm Pic}^dX\times\C^g$, $u_d$ denotes the isomorphism between $  {\rm Pic}^d\times\C^g$ and ${\rm Pic}^{d+1}X\times\C^g$ defined by $(\cl,\nu)\mapsto (\cl([P_0]),\nu)$ and $F$ stands for the isomorphism given by $(x,\cl)\mapsto (x,\cl([x]-[P_0]))$. By commutativity of this diagram, we have $H_{2,d}=u_d \circ p_2 \circ F$ so it suffices to prove that $F^* p_2^* u_d^* \ca_{d+1} \cong \ca_d\boxtimes \cox$. By construction, we have $u_d^* \ca_{d+1}\cong \ca_d$ because the following diagram commutes
$$\xymatrix{
 {\rm Pic}d X \times \C^g\ar[r]^{u_d} & {\rm Pic}^{d+1}X\times \C^g \\
{\rm Pic}^1X \ar[u]_{t_d}  \ar[ru]_{ t_{d+1}} &  }$$

\noindent Moreover, we have
$$\ca_d \boxtimes \cox =p_2^* \ca_d\otimes_{{\cal O}_{X\times {\rm Pic}^d\times \C^g}} p_1^* \cox $$ 
and 
$$p_1^* \cox= p_1^{-1} \cox \otimes_{p_1^{-1}\cox } {\cal O}_{X\times {\rm Pic}^dX\times \C^g} \cong {\cal O} _{X\times {\rm Pic}^dX\times \C^g}$$
hence 
$$\ca_d \boxtimes \cox \cong p_2^* \ca_d \hbox{ (as connections).}$$ 
Finally, let us show that 
$$F^* (\ca_d \boxtimes \cox) \cong \ca_d \boxtimes \cox.$$ 
For this purpose, we now work with the representations. If $\eta$ denotes the representation associated to the sheaf $\ca_d \boxtimes \cox$, the representation associated to $F^* (\ca_d \boxtimes \cox)$ is $\eta \circ F$ and we are now going to prove that these two representations are equal by proving that the two loops $L$ and $F(L)$ are homotopic $\forall L\in \pi_1(X\times {\rm Pic}^dX\times \C^g)$. To simplify notations, we will write the homotopy from the point of view of divisors. Let 
$$t\mapsto L(t)=(x(t),D(t),\nu(t))$$ 
be a loop of $X\times {\rm Pic}^d\times\C^g$. The map $H$ defined by 
$$H(s,t)=(x(t),D(t)+ x(ts)-P_0)$$ 
sends continuously $L$ to  $F(L)$, so 
$$\eta(L)=\eta\circ F(L),~\forall L\in \pi_1(X\times {\rm Pic}^dX\times\C^g).$$ 
This implies that the two connections $\ca_d\boxtimes \cox$ and $F^*(\ca_d\boxtimes \cox)$ are associated to the same representation and are therefore isomorphic. We have thus proved that the connections $\ck_d$ satisfy the Hecke property.
\end{proof}

\begin{rk}
In this example, we consider connections with differential Galois group included in $\G_a\times\G_m$ which is not a reductive group. We are therefore slightly outside the frame of the geometric Langlands correspondence and this might be why only $Hecke_2$ holds in this case.
\end{rk}

The previous proof can be adapted to any connections with Galois group included in a group of the form $\G_a^n\times \G_m^k$, with $k,n$ positive integers. For example, the group of matrices of the type $\begin{pmatrix}a&0&b\cr 0&b&0\cr 0&0&a\end{pmatrix}$ is $\G_m\times \G_m\times \G_a$ whereas the group of matrices of the type $\begin{pmatrix}a&0&b\cr 0&a&c\cr 0&0&a\end{pmatrix}$ is $\G_m\times\G_a\times\G_a$. Note that we only consider commutative groups. For non-commutative groups, the previous method cannot be used.
 
\centerline{Bureau 109, Institut de mathématiques de Toulouse, 118 route de Narbonne, 31062 Toulouse Cedex 9 }
\centerline{{\it email:} poirier@math.ups-tlse.fr}
\end{document}